\documentclass[11pt]{article}
\usepackage[margin=1in]{geometry}
\usepackage{amsmath,amssymb,amsthm}
\usepackage{enumitem}
\usepackage{tikz}
\usepackage{pgfplots}
\pgfplotsset{compat=1.18}
\usetikzlibrary{arrows.meta,positioning,calc}
\usepackage{float}
\usepackage{caption}
\usepackage{hyperref}
\hypersetup{colorlinks=true,linkcolor=blue,citecolor=blue,urlcolor=blue}

\newtheorem{theorem}{Theorem}[section]
\newtheorem{definition}[theorem]{Definition}
\newtheorem{lemma}[theorem]{Lemma}
\newtheorem{proposition}[theorem]{Proposition}
\newtheorem{corollary}[theorem]{Corollary}
\newtheorem{example}[theorem]{Example}
\newtheorem{remark}[theorem]{Remark}
\newtheorem{assumption}[theorem]{Assumption}

\DeclareMathOperator{\Fix}{Fix}
\newcommand{\R}{\mathbb{R}}
\newcommand{\N}{\mathbb{N}}

\title{\bfseries Nadler's Theorem on Generalized Metric Spaces\\ with Weighted Graphs}

\author{
{\large Ismat Beg$^{1}$\qquad Muhammad Aslam Javed$^{2}$}\\[1.2ex]
{\small $^{1}$\,Department of Mathematics \& Statistical Sciences, Lahore School of Economics,}\\
{\small Lahore, Pakistan.\quad\texttt{ibeg@lahoreschool.edu.pk}}\\[0.8ex]
{\small $^{2}$\,Department of Mathematics, Faculty of Science \& Technology, University of Central Punjab,}\\
{\small Lahore, Pakistan \quad\texttt{m.a.javed@ucp.edu.pk}}
}

\date{}

\begin{document}
\maketitle

\begin{abstract}
We introduce a generalized metric space endowed with a weighted graph and establish sufficient conditions for the existence of fixed points for set-valued mappings that satisfy a contractive condition only along the links of the graph. To demonstrate the applicability and significance of our results, several illustrative examples are presented. As an application, we establish the solvability of a class of integral problems involving set-valued data.

\medskip
\noindent\textbf{Keywords:} Fixed point; generalized metric space; weighted graph; set-valued mapping; level cut; rate of convergence.

\noindent\textbf{Mathematics Subject Classifications (2020):} 47H10; 05C22; 47H04; 54H25.
\end{abstract}

\section{Introduction}\label{sec:intro}

Nadler \cite{Nadler1969} generalized the Banach contraction principle and proved that a set-valued mapping with nonempty, closed, bounded values, contractive in the Hausdorff distance, possesses a fixed point. Subsequent developments of this result move along two distinct axes; The first axis weakens the structure carried by the distance. Relaxing the triangle inequality to $d(x,z)\le M\big(d(x,y)+d(y,z)\big)$ with a fixed constant $M\ge1$ yields the $b$-metric spaces of Czerwik \cite{Czerwik1993}, where the set-valued contraction principle survives once the contraction constant is measured against $M$ \cite{Czerwik2022}. Abandoning symmetry, or trading the numerical triangle inequality for qualitative continuity requirements, leads to the \emph{generalized metric spaces} of Reich and Zaslavski \cite{ReichZaslavski2022}, later combined with a graph structure in \cite{ReichZaslavski2026}. We stress that the term ``generalized metric space'' is used in several inequivalent senses in the literature; throughout this paper it denotes the asymmetric, relaxed-triangle structure of Reich and Zaslavski recalled below, and not, for instance, the $\nu$-generalized metric spaces of Suzuki \cite{Suzuki2017}. Parallel efforts replace the metric by a partial metric and build a partial Hausdorff distance suited to the set-valued theory \cite{AydiAbbasVetro2012, BegPathak2018}, or weaken the contractive requirement itself \cite{Vetro2015}.

The second axis reduces the collection of pairs on which the contractive estimate is demanded. Ran and Reurings \cite{RanReurings2004} and Nieto and Rodr\'{i}guez-L\'{o}pez \cite{NietoRodriguezLopez2005} noticed that on a partially ordered metric space it suffices to control distances between comparable elements. Jachymski \cite{Jachymski2008} identified a partial order as a particular directed graph and restated the contraction principle for mappings compatible with an arbitrary graph, thereby subsuming the ordered-space theorems. For set-valued mappings this programme was pursued by Beg, Butt and Radojevi\'{c} \cite{BegButtRadojevic2010} and by Beg and Butt \cite{BegButt2013}, who imposed the Hausdorff contraction only along the edges of a directed graph, by Nicolae, O'Regan and Petru\c{s}el \cite{NicolaeOReganPetrusel2011} for generalized multivalued contractions on a graph, and by Samreen, Kamran and Shahzad \cite{SamreenKamranShahzad2013} in the $b$-metric setting.

Against this background the exact combination treated here -- set-valued mappings, on an \emph{asymmetric} generalized metric space, contractive along the edges of a \emph{weighted} directed graph -- appears not to have been considered. The novelty is not merely the co-occurrence of these three ingredients. The weighting is made to carry quantitative information: the grade of an edge selects the contraction factor applied along it, and because the successive approximations climb into ever more strongly weighted edges, the weights yield a convergence estimate (Proposition~\ref{prop:adaptive}) strictly sharper than any single-threshold bound, and they organize the space into threshold layers whose optimal choice we determine (Proposition~\ref{prop:tradeoff}). Sections~\ref{sec:main}--\ref{sec:applications} establish the theory, the sharpened rate, the trade-off, and three worked models together with the corresponding applications; Section~\ref{sec:conclusion} concludes. 

\section{Generalized Metric Spaces with Weighted Graphs}\label{sec:prelim}

Throughout, $X$ is a nonempty set and $d:X\times X\to[0,\infty)$ satisfies
\begin{equation}\label{eq:zero}
d(x,y)=0\iff x=y .
\end{equation}
Following \cite{ReichZaslavski2022}, we impose instead two conditions relating a distance to its reverse and to composite distances.

\begin{assumption}[symmetry modulus]\label{ass:A1}
For every $\varepsilon>0$ there is $\delta_1(\varepsilon)>0$ such that $d(x,y)\le\delta_1(\varepsilon)$ implies $d(y,x)\le\varepsilon$.
\end{assumption}

\begin{assumption}[triangle modulus]\label{ass:A2}
For every $\varepsilon>0$ there is $\delta_2(\varepsilon)>0$ such that $d(x,y)\le\delta_2(\varepsilon)$ and $d(y,z)\le\delta_2(\varepsilon)$ jointly imply $d(x,z)\le\varepsilon$.
\end{assumption}

The moduli may be chosen non-decreasing with $\delta_i(\varepsilon)\le\varepsilon$. The sets
\[
U(\varepsilon):=\{(x,y)\in X\times X:\ d(x,y)\le\varepsilon\ \text{and}\ d(y,x)\le\varepsilon\},\qquad \varepsilon>0,
\]
form a base for a uniformity on $X$; this uniformity is metrizable and we fix a metric $d_1$ inducing it, assuming throughout that $(X,d_1)$ is complete. By construction $d_1(x_n,x)\to0$ implies $d(x_n,x)\to0$ and $d(x,x_n)\to0$, and conversely. A subset of $X$ is \emph{bounded} when its $d$-diameter is finite. We refer to $(X,d)$, equipped with Assumptions~\ref{ass:A1}--\ref{ass:A2} and completeness of $(X,d_1)$, as a \emph{generalized metric space}.

In the \emph{quantitative regime}, in which this paper works, both moduli are subsumed by a single numerical requirement: there exists $M\ge1$ with
\begin{equation}\label{eq:M}
d(x,z)\ \le\ M\big(d(x,y)+d(y,z)\big)\qquad(x,y,z\in X).
\end{equation}
Setting $y=x$ in \eqref{eq:M} forces $M\ge1$; the inequality yields Assumption~\ref{ass:A2} with $\delta_2(\varepsilon)=\varepsilon/(2M)$, while Assumption~\ref{ass:A1} is retained separately, since \eqref{eq:M} constrains composite distances but not the asymmetry of $d$. When $d$ is moreover symmetric, \eqref{eq:M} is exactly the defining inequality of a $b$-metric space with constant $M$ \cite{Czerwik1993}, and $M=1$ returns an ordinary metric space.

A \emph{weighted graph} on $X$ is a pair $(G,w)$ in which $G$ is a directed graph with vertex set $X$ and edge set $E(G)\subseteq X\times X$ containing every loop $(x,x)$, and $w:E(G)\to(0,1]$ assigns to each edge its \emph{grade}. For $\alpha\in(0,1]$ the \emph{level cut} at $\alpha$ is the directed graph $G_\alpha$ on $X$ whose edge set is
\begin{equation}\label{eq:levelcut}
E_\alpha:=\{e\in E(G):\ w(e)\ge\alpha\}.
\end{equation}
Since $\beta\le\alpha$ gives $E_\alpha\subseteq E_\beta$, the level cuts form a nested decreasing family, illustrated in Figure~\ref{fig:levelcuts}. Write $\widetilde G_\alpha$ for the symmetrization of $G_\alpha$, obtained by disregarding edge orientations. Given $K\subseteq X$ and $x,y\in K$, a \emph{$\widetilde G_\alpha$-chain} in $K$ from $x$ to $y$ is a finite sequence $x=\xi_0,\xi_1,\dots,\xi_q=y$ in $K$ with $(\xi_i,\xi_{i+1})\in E_\alpha$ or $(\xi_{i+1},\xi_i)\in E_\alpha$ for each $i$; its length is $q$. We set
\[
N(\alpha):=\sup_{x,y\in K}\big(\text{shortest }\widetilde G_\alpha\text{-chain length from }x\text{ to }y\big)\in\N\cup\{\infty\},
\quad
A:=\{\alpha\in(0,1]: N(\alpha)<\infty\},
\]
the supremum being $\infty$ when some pair admits no chain. Thus $N(\alpha)$ is the diameter of $\widetilde G_\alpha$ on $K$, and $\alpha\in A$ says that $\widetilde G_\alpha$ is connected on $K$ with uniformly bounded diameter. As $\alpha$ increases, edges are discarded and $N(\alpha)$ is non-decreasing.

\begin{figure}[H]
\centering
\caption{}
\label{fig:levelcuts}
\begin{tikzpicture}[
  every node/.style={circle,draw,inner sep=1.4pt,minimum size=6.5mm,font=\small},
  edg/.style={-{Stealth[length=2mm]},thick},
  gone/.style={-{Stealth[length=2mm]},draw=black!22,thick},
  lab/.style={draw=none,circle=false,inner sep=1pt,font=\scriptsize,fill=white}
]
\begin{scope}
\node (a1) at (0,1.5) {$v_1$};
\node (a2) at (1.8,1.5) {$v_2$};
\node (a3) at (1.8,0) {$v_3$};
\node (a4) at (0,0) {$v_4$};
\node (a5) at (3.5,0.75) {$v_5$};
\draw[edg] (a1)--(a2) node[lab,midway,above]{$0.9$};
\draw[edg] (a2)--(a3) node[lab,midway,right]{$0.6$};
\draw[edg] (a3)--(a4) node[lab,midway,below]{$0.9$};
\draw[edg] (a1)--(a3) node[lab,pos=0.42,above right]{$0.7$};
\draw[edg] (a3)--(a5) node[lab,midway,below right]{$0.4$};
\node[draw=none,font=\small] at (1.7,-0.95) {(a) $(G,w)$};
\end{scope}
\begin{scope}[xshift=5.3cm]
\node (b1) at (0,1.5) {$v_1$};
\node (b2) at (1.8,1.5) {$v_2$};
\node (b3) at (1.8,0) {$v_3$};
\node (b4) at (0,0) {$v_4$};
\node (b5) at (3.5,0.75) {$v_5$};
\draw[edg] (b1)--(b2);
\draw[edg] (b2)--(b3);
\draw[edg] (b3)--(b4);
\draw[edg] (b1)--(b3);
\draw[gone] (b3)--(b5);
\node[draw=none,font=\small] at (1.7,-0.95) {(b) $G_\alpha$, $\alpha=0.5$};
\end{scope}
\begin{scope}[xshift=10.6cm]
\node (c1) at (0,1.5) {$v_1$};
\node (c2) at (1.8,1.5) {$v_2$};
\node (c3) at (1.8,0) {$v_3$};
\node (c4) at (0,0) {$v_4$};
\node (c5) at (3.5,0.75) {$v_5$};
\draw[edg] (c1)--(c2);
\draw[gone] (c2)--(c3);
\draw[edg] (c3)--(c4);
\draw[gone] (c1)--(c3);
\draw[gone] (c3)--(c5);
\node[draw=none,font=\small] at (1.7,-0.95) {(c) $G_\alpha$, $\alpha=0.8$};
\end{scope}
\end{tikzpicture}
\end{figure}

Figure~\ref{fig:levelcuts} shows a weighted graph together with two of its level cuts, loops being suppressed in the drawing. Raising the threshold $\alpha$ deletes the weaker edges, shown faded there: at $\alpha=0.5$ the vertex $v_5$ is severed, and at $\alpha=0.8$ only the two strongest edges survive. Since the family $\{E_\alpha\}$ decreases as $\alpha$ increases, a higher threshold gives a stronger contraction guarantee on a sparser graph.

\begin{remark}\label{rem:Nalpha}
The invariant $N(\alpha)$, and the admissible set $A$, will not be used as hypotheses of the main theorem: that theorem asks only for the local reachability property $(A_\alpha)$ of Definition~\ref{def:propertyA}, which the orbit it constructs supplies along a single chain. Finiteness $N(\alpha)<\infty$ is a convenient \emph{global} sufficient condition guaranteeing that property on all of $K$, and it is exactly the setting of the finite network of Section~\ref{sec:applications}, where it is computable; but it is not necessary, and the trade-off Example~\ref{ex:interval} deliberately operates at a threshold outside $A$ where the global level cut is disconnected yet the orbit still succeeds. We keep $N(\alpha)$ because it measures precisely how the graph thins as the threshold rises, which is the mechanism the weighting exploits.
\end{remark}

Fix a nonempty $K\subseteq X$ closed in $(X,d_1)$, and let $CB(K)$ denote the family of nonempty subsets of $K$ that are closed in $(X,d_1)$ and bounded in $d$. For $A,B\in CB(K)$ and $a\in A$ put $d(a,B):=\inf_{b\in B}d(a,b)$, and define the \emph{excess of $A$ over $B$} by
\begin{equation}\label{eq:excess}
H(A,B)\ :=\ \sup_{a\in A}d(a,B)\ =\ \sup_{a\in A}\inf_{b\in B}d(a,b).
\end{equation}
We deliberately refrain from symmetrizing $H$. Since $d$ itself is not assumed symmetric, and since the contractive estimate below is imposed along a directed edge $(x,y)$, the appropriate set-level quantity is the one-sided excess taken in the direction of that edge rather than a two-sided Hausdorff distance. This is a weaker requirement than a bound on $\max\{H(A,B),H(B,A)\}$; it is what the forward construction of the orbit actually consumes, and, as Proposition~\ref{prop:nonunique} shows, it is genuinely weaker, since an edge $(x,y)$ imposes no constraint at all once $Fx$ already reaches into $Fy$. The next elementary principle is the only selection device the construction requires.

\begin{lemma}[approximate selection]\label{lem:selection}
Let $A,B\in CB(K)$, $a\in A$ and $\gamma>0$. Then there exists $b\in B$ with
\[
d(a,b)\ \le\ H(A,B)+\gamma .
\]
\end{lemma}

\begin{proof}
By the definition of the supremum in \eqref{eq:excess} we have $d(a,B)\le H(A,B)$, and by the definition of the infimum defining $d(a,B)$ there is $b\in B$ with $d(a,b)\le d(a,B)+\gamma$. Combining the two estimates gives the assertion.
\end{proof}

Lemma~\ref{lem:selection} requires no compactness and uses nothing about $d$ beyond non-negativity; it operates on the full class $CB(K)$, which is exactly the class of values admitted in the classical set-valued theory. This is worth recording, since the analogous statement is not available in every enlargement of the metric framework: the corresponding fuzzy construction, for instance, is known to satisfy the required triangle-type estimate only on the smaller family of compact values.

\section{Main Results}\label{sec:main}

Throughout this section $(X,d)$ is a generalized metric space obeying \eqref{eq:M}, $(G,w)$ is a weighted graph on $X$, and $K\subseteq X$ is nonempty and closed in $(X,d_1)$.

\begin{definition}\label{def:compatible}
A mapping $F:K\to CB(K)$ is \emph{grade-compatible} with $(G,w)$ if, whenever $(x,y)\in E(G)$ with $x,y\in K$ and $y\in F(x)$, then for every $\gamma>0$ and every $v\in F(y)$ satisfying $d(y,v)\le H(Fx,Fy)+\gamma$ one has
\[
(y,v)\in E(G)\qquad\text{and}\qquad w(y,v)\ \ge\ w(x,y).
\]
\end{definition}

Definition~\ref{def:compatible} is the set-valued form of the edge-preservation requirement of \cite{Jachymski2008}, reinforced by the monotonicity clause $w(y,v)\ge w(x,y)$. The clause forbids a step from weakening the link it travels along, and consequently confines an entire sequence of successive selections to a single level cut; moreover it forces the grades encountered along that sequence to be non-decreasing, a fact that Proposition~\ref{prop:adaptive} converts into a sharpened rate. It is formulated for precisely the configuration produced by the construction below, namely the terminal vertex $y$ of an edge that already belongs to $F(x)$ together with a point $v\in F(y)$ selected near $y$; a wider quantification would contribute nothing used in the proofs.

\begin{definition}\label{def:contraction}
Let $\{\lambda_\alpha\}_{\alpha\in(0,1]}\subseteq[0,1)$ be non-increasing in $\alpha$, so that $\alpha\le\beta$ implies $\lambda_\beta\le\lambda_\alpha$ and a higher grade never carries a weaker contraction. A grade-compatible mapping $F:K\to CB(K)$ is a \emph{graded set-valued contraction} with respect to $(G,w)$ and $\{\lambda_\alpha\}$ if
\begin{equation}\label{eq:contraction}
H(Fx,Fy)\ \le\ \lambda_{w(x,y)}\,d(x,y)\qquad\text{for every }(x,y)\in E(G)\text{ with }x,y\in K.
\end{equation}
\end{definition}

If the graph is complete and every edge is graded $1$, condition \eqref{eq:contraction} reduces to the requirement $H(Fx,Fy)\le\lambda_1 d(x,y)$ for all pairs, that is, to the classical contractive estimate for the one-sided excess. The grading refines this by letting the contraction factor follow the strength of the edge.

\begin{remark}\label{rem:external}
For a single-valued self-map $T$ a grading may be synthesized from the map itself through the ratios $d(T^{n+1}x,T^{n+1}y)/d(T^{n}x,T^{n}y)$ of consecutive iterates, since $T^{n}(x)$ is a well-defined point for every $n$. A genuinely set-valued $F$ admits no canonical iterate $F^{n}(x)$: each stage involves a choice of a point in a set, and distinct choices generate distinct ratios at the same pair. The grading of a graded set-valued contraction is therefore data supplied alongside the mapping, on the same footing as the single contraction constant of the classical theorem, and is not recoverable from $F$ alone.
\end{remark}

\begin{definition}[property $(A_\alpha)$]\label{def:propertyA}
The weighted graph $(G,w)$ has \emph{property $(A_\alpha)$} on $K$ if every sequence $\{x_n\}\subseteq K$ with $(x_n,x_{n+1})\in E_\alpha$ for all $n$ and $x_n\to x\in K$ in $(X,d_1)$ possesses a subsequence $\{x_{n_k}\}$ with $(x_{n_k},x)\in E_\alpha$ for every $k$.
\end{definition}

\begin{theorem}\label{thm:main}
Let $(X,d)$ be a generalized metric space satisfying \eqref{eq:M} with $(X,d_1)$ complete, let $(G,w)$ be a weighted graph on $X$, and let $K\subseteq X$ be nonempty and closed in $(X,d_1)$. Let $F:K\to CB(K)$ be a graded set-valued contraction with respect to $(G,w)$ and $\{\lambda_\alpha\}$, and let $\alpha\in(0,1]$ be such that
\begin{equation}\label{eq:hyp}
M\lambda_\alpha<1
\end{equation}
and $(G,w)$ has property $(A_\alpha)$ on $K$. Then for every $x_0\in K$ admitting $x_1\in F(x_0)$ with $(x_0,x_1)\in E_\alpha$ there exists $x^\ast\in K$ with
\[
x^\ast\in F(x^\ast),
\]
obtained as the limit in $(X,d_1)$ of a sequence of successive selections issuing from $x_0$. Moreover, for every $\lambda'\in(\lambda_\alpha,1/M)$ there is a constant $C>0$, depending only on $d(x_0,x_1)$, $\lambda_\alpha$, $\lambda'$ and $M$, such that the constructed sequence obeys
\begin{equation}\label{eq:rate}
\sup_{m\ge n}\,d(x_n,x_m)\ \le\ C\,(\lambda')^{\,n}\qquad(n\ge0).
\end{equation}
\end{theorem}

\begin{proof}
Write $\lambda:=\lambda_\alpha$ and fix $\beta_0>0$. If $\lambda=0$ the construction below gives $d(x_1,x_2)\le\beta_0$ with $\beta_0$ arbitrary and all subsequent estimates hold trivially, so assume $\lambda\in(0,1)$.

\emph{Construction.} Put $a_1:=d(x_0,x_1)$. Suppose $x_{n-1},x_n$ have been constructed with $(x_{n-1},x_n)\in E_\alpha$ and $x_n\in F(x_{n-1})$. Applying Lemma~\ref{lem:selection} with $A=F(x_{n-1})$, $B=F(x_n)$, $a=x_n$ and slack $\gamma_n:=\beta_0\lambda^{\,n-1}$ produces $x_{n+1}\in F(x_n)$ with
\[
a_{n+1}:=d(x_n,x_{n+1})\ \le\ H\big(F(x_{n-1}),F(x_n)\big)+\gamma_n .
\]
Grade-compatibility (Definition~\ref{def:compatible}) gives $(x_n,x_{n+1})\in E(G)$ with $w(x_n,x_{n+1})\ge w(x_{n-1},x_n)$; since $w(x_0,x_1)\ge\alpha$, induction yields $w(x_n,x_{n+1})\ge\alpha$ for every $n$, so that every edge of the sequence lies in $E_\alpha$. As $w(x_{n-1},x_n)\ge\alpha$ and $\{\lambda_\alpha\}$ is non-increasing, we have $\lambda_{w(x_{n-1},x_n)}\le\lambda$, and \eqref{eq:contraction} produces the linear recursion
\begin{equation}\label{eq:recursion}
a_{n+1}\ \le\ \lambda a_n+\beta_0\lambda^{\,n-1}\qquad(n\ge1).
\end{equation}

\emph{Step 1: geometric control of consecutive gaps.} Unwinding \eqref{eq:recursion} inductively gives
\begin{equation}\label{eq:gap}
a_n\ \le\ \lambda^{\,n-1}a_1+(n-1)\beta_0\lambda^{\,n-2}\qquad(n\ge1),
\end{equation}
the second summand being read as $0$ for $n=1$. Hypothesis \eqref{eq:hyp} gives $\lambda<1/M$, so we may fix $\lambda'\in(\lambda,1/M)$. Since $(\lambda/\lambda')^{\,n-1}\to0$ and $(n-1)(\lambda/\lambda')^{\,n-2}\to0$ as $n\to\infty$, both terms on the right of \eqref{eq:gap} are $O\big((\lambda')^{\,n-1}\big)$, whence there is $C_1>0$, depending only on $a_1,\beta_0,\lambda,\lambda'$, with
\begin{equation}\label{eq:gap2}
a_n\ \le\ C_1(\lambda')^{\,n-1}\qquad(n\ge1).
\end{equation}

\emph{Step 2: a bound uniform in the chain length.} We claim that
\begin{equation}\label{eq:uniform}
d(x_n,x_{n+p})\ \le\ C(\lambda')^{\,n}\qquad(n\ge0,\ p\ge1),\qquad C:=\frac{MC_1}{1-M\lambda'} .
\end{equation}
Note $C>0$ because $M\lambda'<1$, and $C\ge C_1$ because $M\ge1$ forces $M\ge1-M\lambda'$. We induct on $p$. For $p=1$, estimate \eqref{eq:gap2} gives $d(x_n,x_{n+1})=a_{n+1}\le C_1(\lambda')^{n}\le C(\lambda')^{n}$. Assume \eqref{eq:uniform} holds at chain length $p$ for every starting index. Using \eqref{eq:M} with midpoint $x_{n+1}$, then \eqref{eq:gap2} and the inductive hypothesis,

\[
\begin{aligned}
d(x_n,x_{n+p+1}) &\le M\big[d(x_n,x_{n+1})+d(x_{n+1},x_{n+1+p})\big] \\
&\le M\big[C_1(\lambda')^{n}+C(\lambda')^{n+1}\big] = M(\lambda')^{n}\big(C_1+C\lambda'\big)
\end{aligned}
\]
By the choice of $C$ we have $M(C_1+C\lambda')=MC_1+MC\lambda'=C(1-M\lambda')+MC\lambda'=C$, so $d(x_n,x_{n+p+1})\le C(\lambda')^{n}$. This closes the induction and establishes \eqref{eq:uniform}, hence also \eqref{eq:rate}.

\emph{Step 3: convergence.} By \eqref{eq:uniform}, $d(x_n,x_{n+p})\to0$ as $n\to\infty$ uniformly in $p$; Assumption~\ref{ass:A1} transfers this to $d(x_{n+p},x_n)$, so $\{x_n\}$ is Cauchy in $(X,d_1)$. Completeness of $(X,d_1)$ and closedness of $K$ provide $x^\ast\in K$ with $x_n\to x^\ast$ in $(X,d_1)$, and therefore $d(x_n,x^\ast)\to0$ and $d(x^\ast,x_n)\to0$.

\emph{Step 4: the limit lies in its own image.} Property $(A_\alpha)$ yields a subsequence with $(x_{n_k},x^\ast)\in E_\alpha$ for all $k$, so \eqref{eq:contraction} gives $H(Fx_{n_k},Fx^\ast)\le\lambda\,d(x_{n_k},x^\ast)\to0$. Let $\eta>0$. For $k$ large we have $H(Fx_{n_k},Fx^\ast)<\eta$, and since $x_{n_k+1}\in F(x_{n_k})$, Lemma~\ref{lem:selection} applied with $a=x_{n_k+1}$ and slack $\eta$ supplies $z_k\in F(x^\ast)$ with
\[
d(x_{n_k+1},z_k)\ \le\ H(Fx_{n_k},Fx^\ast)+\eta\ <\ 2\eta .
\]
Since $x_{n_k+1}\to x^\ast$ we also have $d(x^\ast,x_{n_k+1})\to0$; combining a small $d(x^\ast,x_{n_k+1})$ with a small $d(x_{n_k+1},z_k)$ through Assumption~\ref{ass:A2} renders $d(x^\ast,z_k)$ arbitrarily small, and Assumption~\ref{ass:A1} then renders $d(z_k,x^\ast)$ arbitrarily small as well. Hence $z_k\to x^\ast$ in $(X,d_1)$, and as $z_k\in F(x^\ast)$ with $F(x^\ast)$ closed in $(X,d_1)$ we conclude $x^\ast\in F(x^\ast)$.
\end{proof}

\begin{remark}\label{rem:exact}
When the values of $F$ are compact the infimum in \eqref{eq:excess} is attained, the slack $\beta_0$ may be dispensed with, and \eqref{eq:recursion} sharpens to $a_{n+1}\le\lambda_{w(x_{n-1},x_n)}\,a_n\le\lambda_\alpha a_n$, whence
\begin{equation}\label{eq:exactgap}
a_n\ \le\ \lambda_\alpha^{\,n-1}a_1\qquad(n\ge1).
\end{equation}
Estimate \eqref{eq:exactgap} is the single-threshold envelope plotted in Figure~\ref{fig:convergence}; Proposition~\ref{prop:adaptive} sharpens it.
\end{remark}

\begin{figure}[H]
\centering
\caption{}
\label{fig:convergence}
\begin{tikzpicture}
\begin{semilogyaxis}[
  width=0.52\textwidth, height=6.2cm,
  xlabel={iteration $n$}, ylabel={gap $a_n$},
  xmin=0.8, xmax=5.2, ymin=1e-32, ymax=1e1,
  xtick={1,2,3,4,5},
  ytick={1e0,1e-5,1e-10,1e-15,1e-20,1e-25,1e-30},
  legend style={at={(0.03,0.03)},anchor=south west,font=\scriptsize,draw=black!30},
  grid=both, grid style={draw=black!12},
  label style={font=\small}, tick label style={font=\scriptsize}
]
\addplot[black!45,dashed,thick,mark=square,mark size=1.6pt] coordinates
 {(1,1.296e-1)(2,5.4756e-2)(3,2.313441e-2)(4,9.774288e-3)(5,4.129637e-3)};
\addlegendentry{single-thr., $\alpha=0.35$}
\addplot[black,dashed,thick,mark=triangle,mark size=2.2pt] coordinates
 {(1,1.296e-1)(2,6.27264e-3)(3,3.035958e-4)(4,1.469404e-5)(5,7.111913e-7)};
\addlegendentry{single-thr., $\alpha^{\ast}=0.78$}
\addplot[black,thick,mark=o,mark size=2.2pt,densely dotted] coordinates
 {(1,1.296e-1)(2,6.273e-3)(3,2.559e-6)(4,1.024e-13)(5,4.096e-29)};
\addlegendentry{grade-adaptive}
\addplot[black,thick,mark=*,mark size=1.7pt] coordinates
 {(1,1.296e-1)(2,1.56816e-3)(3,1.59968e-7)(4,1.6e-15)(5,1.6e-31)};
\addlegendentry{observed $a_n$}
\end{semilogyaxis}
\end{tikzpicture}
\hfill
\begin{tikzpicture}
\begin{axis}[
  width=0.44\textwidth, height=6.2cm,
  xlabel={iteration $n$}, ylabel={grade $g_n=w(x_{n-1},x_n)$},
  xmin=0.8, xmax=5.2, ymin=0.7, ymax=1.04,
  xtick={1,2,3,4,5},
  grid=both, grid style={draw=black!12},
  label style={font=\small}, tick label style={font=\scriptsize}
]
\addplot[black,thick,mark=*,mark size=1.7pt] coordinates
 {(1,0.7800)(2,0.9798)(3,0.9998)(4,1.0000)(5,1.0000)};
\addplot[black!45,dashed,thick] coordinates {(0.8,0.78)(5.2,0.78)};
\node[font=\scriptsize,anchor=west] at (axis cs:1.15,0.755) {entry grade $\alpha^{\ast}=0.78$};
\end{axis}
\end{tikzpicture}
\end{figure}

\begin{remark}\label{rem:unbounded}
No boundedness of $K$ is required in Theorem~\ref{thm:main}. Only the values $F(x)$ are asked to be bounded, through membership in $CB(K)$; the diameter of $K$ itself never enters the estimates. This matters for the application of Section~\ref{sec:applications}, where $K$ is an unbounded function space.
\end{remark}

\begin{corollary}[termination on finite networks]\label{cor:finite}
If, in addition to the hypotheses of Theorem~\ref{thm:main}, $K$ is finite with at least two points, then the constructed sequence is eventually constant: there is a finite index beyond which $x_{n+1}=x_n=x^\ast$.
\end{corollary}

\begin{proof}
Let $\rho:=\min\{d(x,y):x,y\in K,\ x\ne y\}$, positive by \eqref{eq:zero} and the finiteness of $K$. By \eqref{eq:gap2}, $a_n\to0$, so $a_n<\rho$ for all large $n$, which by definition of $\rho$ forces $x_{n-1}=x_n$ from that index onwards; the common value is the limit $x^\ast$ obtained in Step~3.
\end{proof}

The rate \eqref{eq:rate} of Theorem~\ref{thm:main} is governed by the single constant $\lambda_\alpha$ attached to the chosen threshold, and in that estimate the weighting plays no quantitative role. The next result shows that the weighting is not, however, inert: because grade-compatibility forces the grades met along the orbit to ascend, the contraction factors actually applied decrease, and the resulting bound can be far tighter than \eqref{eq:exactgap}. Here the weights enter the rate directly, through the whole family $\{\lambda_\alpha\}$ evaluated along the orbit rather than at a single point.

\begin{proposition}[grade-adaptive rate]\label{prop:adaptive}
Assume the hypotheses of Theorem~\ref{thm:main}, and suppose in addition that $F$ has compact values, so that the exact selection of Remark~\ref{rem:exact} is available. Let $(x_n)$ be the orbit constructed from $x_0$, write $a_n:=d(x_{n-1},x_n)$ and let
\[
g_k:=w(x_{k-1},x_k)\qquad(k\ge1)
\]
be the grade of its $k$-th edge. Then $g_1\le g_2\le\cdots$, all $g_k\ge\alpha$, and
\begin{equation}\label{eq:adaptive}
a_{n+1}\ \le\ \Big(\prod_{k=1}^{n}\lambda_{g_k}\Big)\,a_1\ \le\ \lambda_\alpha^{\,n}\,a_1\qquad(n\ge0),
\end{equation}
the empty product being $1$. The first inequality is strict relative to the second whenever some $g_k>\alpha$ with $\lambda_{g_k}<\lambda_\alpha$; in particular, if $\lambda_{g_k}=0$ for some $k$, the orbit reaches its limit in finitely many steps.
\end{proposition}

\begin{proof}
That the grades are non-decreasing and bounded below by $\alpha$ was shown in the proof of Theorem~\ref{thm:main}. With compact values, exact selection gives $a_{k+1}=d(x_k,x_{k+1})\le H(Fx_{k-1},Fx_k)$, and \eqref{eq:contraction} bounds the right-hand side by $\lambda_{w(x_{k-1},x_k)}\,d(x_{k-1},x_k)=\lambda_{g_k}a_k$. Iterating $a_{k+1}\le\lambda_{g_k}a_k$ from $k=1$ to $n$ yields the first inequality in \eqref{eq:adaptive}. Since each $g_k\ge\alpha$ and $\{\lambda_\alpha\}$ is non-increasing, $\lambda_{g_k}\le\lambda_\alpha$, giving the second. If some $\lambda_{g_k}=0$ then $a_{k+1}=0$, so $x_k=x_{k+1}$ is the limit and the orbit is thereafter constant.
\end{proof}

\begin{remark}\label{rem:adaptive-role}
Proposition~\ref{prop:adaptive} is the precise sense in which the graph's weights carry quantitative information about convergence, not merely about which edges are admissible. The plain rate \eqref{eq:exactgap} sees only the entry grade $\alpha$; the adaptive rate \eqref{eq:adaptive} sees the whole ascending profile $g_1\le g_2\le\cdots$ and is strictly sharper as soon as the orbit climbs. Example~\ref{ex:interval} and the finite network of Section~\ref{sec:applications} exhibit orbits whose grades rise to $1$, where \eqref{eq:adaptive} outruns \eqref{eq:exactgap} by many orders of magnitude and pinpoints exact termination that \eqref{eq:exactgap} cannot detect.
\end{remark}

The parameter $\alpha$ in Theorem~\ref{thm:main} is at the disposal of the user, and the two quantities it governs move in opposite directions. Raising $\alpha$ lowers $\lambda_\alpha$ and hence sharpens the guaranteed rate \eqref{eq:rate}, but it simultaneously deletes edges and therefore shrinks the supply of admissible initial points. We make this precise. For $\alpha\in(0,1]$ set
\begin{equation}\label{eq:XFalpha}
X_F^{\alpha}\ :=\ \{x\in K:\ (x,u)\in E_\alpha\ \text{for some}\ u\in F(x)\},
\end{equation}
the set of points from which the construction of Theorem~\ref{thm:main} can be launched at threshold $\alpha$, and let
\begin{equation}\label{eq:Ax0}
\mathcal A(x_0)\ :=\ \{\alpha\in(0,1]:\ M\lambda_\alpha<1,\ (G,w)\ \text{has property}\ (A_\alpha)\ \text{on}\ K,\ x_0\in X_F^{\alpha}\}.
\end{equation}

\begin{proposition}\label{prop:tradeoff}
Let $F$ be a graded set-valued contraction with respect to $(G,w)$ and $\{\lambda_\alpha\}$. Then:
\begin{enumerate}[label=\textnormal{(\roman*)}]
\item the map $\alpha\mapsto X_F^{\alpha}$ is non-increasing: $\alpha\le\beta$ implies $X_F^{\beta}\subseteq X_F^{\alpha}$;
\item the map $\alpha\mapsto\lambda_\alpha$ is non-increasing, so the rate guaranteed by \eqref{eq:rate} improves as $\alpha$ grows;
\item for every $x_0\in K$ with $\mathcal A(x_0)\ne\emptyset$ and every $\alpha\in\mathcal A(x_0)$, the conclusion of Theorem~\ref{thm:main} holds from $x_0$ with rate governed by $\lambda_\alpha$; consequently the sharpest single-threshold guarantee available at $x_0$ is obtained at
\[
\alpha^{\ast}(x_0)\ :=\ \sup\mathcal A(x_0),
\]
whenever this supremum belongs to $\mathcal A(x_0)$.
\end{enumerate}
\end{proposition}

\begin{proof}
(i) If $\alpha\le\beta$ then $E_\beta\subseteq E_\alpha$ by \eqref{eq:levelcut}, so any $x$ possessing $u\in F(x)$ with $(x,u)\in E_\beta$ also has $(x,u)\in E_\alpha$. (ii) is the monotonicity postulated in Definition~\ref{def:contraction}, and the rate \eqref{eq:rate} is admissible for every $\lambda'\in(\lambda_\alpha,1/M)$, an interval which widens downwards as $\lambda_\alpha$ decreases. (iii) Each $\alpha\in\mathcal A(x_0)$ satisfies by construction all the hypotheses of Theorem~\ref{thm:main} together with the existence of a starting edge at $x_0$, so the theorem applies at that threshold; among these the smallest $\lambda_\alpha$, hence by (ii) the largest admissible $\alpha$, gives the sharpest bound.
\end{proof}

Proposition~\ref{prop:tradeoff} governs the choice of entry threshold, and Proposition~\ref{prop:adaptive} then improves upon whatever that choice guarantees, since once the orbit is launched its grades can only rise. Passing to an ungraded model amounts to fixing once and for all the smallest grade occurring in the graph, which by (i)--(ii) maximizes the supply of starting points while simultaneously choosing the worst available contraction factor and discarding the adaptive refinement entirely. Example~\ref{ex:interval} and the finite network of Section~\ref{sec:applications} exhibit data for which this ungraded choice is strictly and substantially inferior, and Figure~\ref{fig:tradeoff} displays the competition graphically.

\begin{figure}[H]
\centering
\caption{}
\label{fig:tradeoff}
\begin{tikzpicture}
\begin{axis}[
  width=0.62\textwidth, height=6.2cm,
  xlabel={threshold $\alpha$},
  xmin=0.29, xmax=1.0, ymin=0, ymax=0.9,
  legend style={at={(0.98,0.98)},anchor=north east,font=\scriptsize,draw=black!30},
  grid=both, grid style={draw=black!12},
  label style={font=\small}, tick label style={font=\scriptsize}
]
\addplot[black,thick] coordinates
 {(0.2929,0.5000)(0.35,0.4225)(0.45,0.3025)(0.52,0.2304)(0.60,0.1600)(0.70,0.0900)(0.78,0.0484)(0.85,0.0225)(0.90,0.0100)(0.95,0.0025)(1.0,0.0)};
\addlegendentry{contraction factor $\lambda_\alpha=(1-\alpha)^2$}
\addplot[black,dashed,thick] coordinates
 {(0.2929,0.8000)(0.35,0.8000)(0.45,0.8000)(0.52,0.8000)(0.60,0.6833)(0.70,0.5298)(0.78,0.4000)(0.85,0.2804)(0.90,0.1909)(0.95,0.0976)(1.0,0.0)};
\addlegendentry{largest admissible initial point $\sup X_F^{\alpha}$}
\addplot[black!55,dotted,thick] coordinates {(0.78,0)(0.78,0.9)};
\node[font=\scriptsize,anchor=west,rotate=90] at (axis cs:0.788,0.30) {$\alpha^{\ast}(0.4)=0.78$};
\end{axis}
\end{tikzpicture}
\end{figure}

Limit point is not unique. This is not through any degeneracy of the hypotheses: the contractive condition constrains the distance between two \emph{images} and says nothing about the size of an individual image.

\begin{proposition}\label{prop:nonunique}
There exist a generalized metric space satisfying \eqref{eq:M}, a weighted graph $(G,w)$ whose level cut at the relevant threshold is connected, and a non-constant graded set-valued contraction $F$ satisfying \eqref{eq:hyp} with $\lambda_\alpha>0$, whose fixed point set $\Fix F:=\{x: x\in F(x)\}$ contains two distinct points.
\end{proposition}

\begin{proof}
Let $K=\{0,2\}\subseteq\R$ with $d(x,y)=|x-y|$, so \eqref{eq:M} holds with $M=1$. Let $E(G)$ consist of the two loops together with the single directed edge $(0,2)$, let $w\equiv1$, and set $\lambda_1:=\tfrac12$. Define
\[
F(0):=\{0\},\qquad F(2):=\{0,2\},
\]
so $F$ is not constant. On the loops, $H(F0,F0)=H(F2,F2)=0=\lambda_1\cdot 0$. On the edge $(0,2)$,
\[
H(F0,F2)=\inf_{v\in\{0,2\}}d(0,v)=0\ \le\ \tfrac12\cdot 2=\lambda_1 d(0,2),
\]
so \eqref{eq:contraction} holds. Grade-compatibility is verified on the two configurations to which it applies: for $(0,0)$ with $0\in F(0)$ the only admissible selection is $v=0$ and $(0,0)\in E(G)$; for $(2,2)$ with $2\in F(2)$ and slack $\gamma<2$ the only admissible selection is $v=2$ and $(2,2)\in E(G)$; in both cases the grade is preserved. The configuration $(0,2)$ is vacuous since $2\notin F(0)$. Here $E_1=E(G)$, whose symmetrization joins $0$ and $2$, so $N(1)=1$ and property $(A_1)$ holds trivially on the finite discrete set $K$. Finally $M\lambda_1=\tfrac12<1$. Since $0\in F(0)$ and $2\in F(2)$, we obtain $\Fix F=\{0,2\}$.
\end{proof}

\begin{example}[a finite weighted network]\label{ex:finite}
Let $X=\R$ with $d(x,y)=|x-y|$, so \eqref{eq:M} holds with $M=1$, and let $K=\{0,1,3,9\}$. Define $F:K\to CB(K)$ by
\[
F(9)=\{1,3\},\qquad F(3)=\{0,1\},\qquad F(1)=\{0\},\qquad F(0)=\{0\},
\]
and let $E(G)$ consist of the loops together with the pairs $\{9,3\}$, $\{3,1\}$, $\{1,0\}$ and their reversals, with grades
\[
w(9,3)=w(3,9)=\tfrac12,\qquad w(3,1)=w(1,3)=\tfrac12,\qquad w(1,0)=w(0,1)=1,
\]
every loop being graded $1$. Take $\lambda_{1/2}=\tfrac12$ and $\lambda_1=0$, an admissible family since $\lambda_1\le\lambda_{1/2}$. The weighted graph and the action of $F$ are displayed in Figure~\ref{fig:finitegraph}, in which solid arrows are graph edges carrying the indicated grades and dashed arrows show the images $F(x)$; panels (b) and (c) of that figure display the level cuts, at $\alpha=\tfrac12$ the cut being the connected path of diameter $3$, whereas at $\alpha=1$ only the strongest edge survives and the cut is disconnected, so $1\notin A$ and the threshold $\tfrac12$ is the operative one. Table~\ref{tab:example} verifies the contractive estimate \eqref{eq:contraction} on every non-loop edge of this example.

\begin{figure}[H]
\centering
\caption{}
\label{fig:finitegraph}
\begin{tikzpicture}[
  every node/.style={circle,draw,inner sep=1.4pt,minimum size=7.5mm,font=\small},
  edg/.style={-{Stealth[length=2mm]},thick},
  gone/.style={-{Stealth[length=2mm]},draw=black!25,thick},
  img/.style={-{Stealth[length=2mm]},dashed,draw=black!60},
  lab/.style={draw=none,circle=false,inner sep=1pt,font=\scriptsize,fill=white}
]
\begin{scope}
\node (n9) at (0,0) {$9$};
\node (n3) at (1.9,0) {$3$};
\node (n1) at (3.8,0) {$1$};
\node (n0) at (5.7,0) {$0$};
\draw[edg] (n9)--(n3) node[lab,midway,above]{$\tfrac12$};
\draw[edg] (n3)--(n1) node[lab,midway,above]{$\tfrac12$};
\draw[edg] (n1)--(n0) node[lab,midway,above]{$1$};
\draw[img] (n9) to[bend right=38] (n1);
\draw[img] (n9) to[bend right=25] (n3);
\draw[img] (n3) to[bend right=32] (n1);
\draw[img] (n3) to[bend right=42] (n0);
\draw[img] (n1) to[bend right=28] (n0);
\node[draw=none,font=\small] at (2.85,-1.85) {(a) $(G,w)$ with the action of $F$ (dashed)};
\end{scope}
\begin{scope}[yshift=-3.3cm]
\node (m9) at (0,0) {$9$};
\node (m3) at (1.9,0) {$3$};
\node (m1) at (3.8,0) {$1$};
\node (m0) at (5.7,0) {$0$};
\draw[edg] (m9)--(m3);
\draw[edg] (m3)--(m1);
\draw[edg] (m1)--(m0);
\node[draw=none,font=\small] at (2.85,-1.0) {(b) $G_\alpha$, $\alpha=\tfrac12$:\, $N(\tfrac12)=3$};
\end{scope}
\begin{scope}[yshift=-3.3cm,xshift=7.6cm]
\node (p9) at (0,0) {$9$};
\node (p3) at (1.9,0) {$3$};
\node (p1) at (3.8,0) {$1$};
\node (p0) at (5.7,0) {$0$};
\draw[gone] (p9)--(p3);
\draw[gone] (p3)--(p1);
\draw[edg] (p1)--(p0);
\node[draw=none,font=\small] at (2.85,-1.0) {(c) $G_\alpha$, $\alpha=1$:\, $N(1)=\infty$};
\end{scope}
\end{tikzpicture}
\end{figure}

\begin{table}[H]
\centering
\caption{Verification of the contractive estimate}
\label{tab:example}
\begin{tabular}{c c c c c}
\hline
edge $(x,y)$ & $F(x),\,F(y)$ & $H(Fx,Fy)$ & $d(x,y)$ & bound $\lambda_{w(x,y)}\,d(x,y)$\\
\hline
$(9,3)$ & $\{1,3\},\,\{0,1\}$ & $2$ & $6$ & $\tfrac12\cdot6=3$\\
$(3,1)$ & $\{0,1\},\,\{0\}$ & $1$ & $2$ & $\tfrac12\cdot2=1$\\
$(1,0)$ & $\{0\},\,\{0\}$ & $0$ & $1$ & $0\cdot1=0$\\
\hline
\end{tabular}
\end{table}

Thus $F$ is a graded set-valued contraction. Grade-compatibility must be checked on every edge $(x,y)$ with $y\in F(x)$, not merely along one trajectory; the configurations that are not vacuous are $(9,3)$, $(3,1)$, $(1,0)$ and the loop $(0,0)$, and on each the admissible successor lies on an edge of grade at least that of the incoming edge, the grades $\tfrac12,\tfrac12,1,1$ being non-decreasing. The level cut $\widetilde G_{1/2}$ is the connected path $9-3-1-0$, so $N(\tfrac12)=3$ and $\tfrac12\in A$, while $\widetilde G_{1}$ is disconnected; property $(A_{1/2})$ holds trivially on the finite discrete set $K$, and $M\lambda_{1/2}=\tfrac12<1$. Starting from $x_0=9$ with $x_1=3\in F(9)$ and $(9,3)\in E_{1/2}$, every image above is attained exactly, so the selections are forced and
\[
9\ \longrightarrow\ 3\ \longrightarrow\ 1\ \longrightarrow\ 0\ \longrightarrow\ 0\ \longrightarrow\cdots,
\]
reaching the fixed point $x^\ast=0\in F(0)$ in three steps. The orbit grades are $g_1=\tfrac12,g_2=\tfrac12,g_3=1$, so the grade-adaptive bound \eqref{eq:adaptive} gives $a_4\le\lambda_{1/2}\lambda_{1/2}\lambda_1\,a_1=0$, correctly predicting termination at the third step, whereas the single-threshold bound \eqref{eq:exactgap} yields only $a_4\le\lambda_{1/2}^{3}a_1=\tfrac34$.
\end{example}

\begin{example}[an interval, with an interior optimal threshold]\label{ex:interval}
Let $X=\R$, $d(x,y)=|x-y|^{2}$ and $K=[0,\tfrac45]$. By the convexity of $r\mapsto r^{2}$ we have $|x-z|^{2}\le2\big(|x-y|^{2}+|y-z|^{2}\big)$, so \eqref{eq:M} holds with
\[
M=2 ,
\]
and $(X,d)$ is a genuine $b$-metric space, not an ordinary metric space; the induced uniformity is the Euclidean one, so $d_1$ may be taken to be the Euclidean metric and $(X,d_1)$ is complete, with $K$ closed. Define
\[
F(x):=\Big[\tfrac{x^{2}}{4},\ \tfrac{x^{2}}{2}\Big]\qquad(x\in K),
\]
a nonempty compact subinterval of $[0,\tfrac{8}{25}]\subseteq K$, so $F:K\to CB(K)$. Grade the complete graph on $K$ by
\[
w(x,y):=1-\tfrac{x+y}{2}\ \in\ \big[\tfrac15,1\big],
\]
so that by \eqref{eq:levelcut}
\[
E_\alpha=\{(x,y)\in K\times K:\ x+y\le2(1-\alpha)\}.
\]
For $x>y$ the excess is computed from the right endpoints,
\[
\sup_{u\in F(x)}\inf_{v\in F(y)}|u-v|=\tfrac{x^{2}-y^{2}}{2}=\tfrac{x+y}{2}\,|x-y| ,
\]
whence, squaring,
\[
H(Fx,Fy)=\Big(\tfrac{x+y}{2}\Big)^{2}d(x,y)\ \le\ (1-\alpha)^{2}\,d(x,y)\qquad\text{for }(x,y)\in E_\alpha .
\]
Therefore \eqref{eq:contraction} holds with
\[
\lambda_\alpha:=(1-\alpha)^{2},
\]
a family decreasing in $\alpha$ as required. Grade-compatibility holds because $y\in F(x)$ forces $y\le x^{2}/2\le x/2$ and then $v\in F(y)$ forces $v\le y^{2}/2\le y/2\le x/4<x$, so $w(y,v)=1-\tfrac{y+v}{2}\ge1-\tfrac{x+y}{2}=w(x,y)$. Condition \eqref{eq:hyp} reads $2(1-\alpha)^{2}<1$, that is
\[
\alpha\ >\ 1-\tfrac{1}{\sqrt2}\ \approx\ 0.2929 .
\]
Choosing at each stage the left endpoint $x_{n+1}=x_n^{2}/4\in F(x_n)$ gives $(x_n,x_{n+1})\in E_\alpha$ precisely when $x_n+x_n^{2}/4\le2(1-\alpha)$, so by \eqref{eq:XFalpha}
\[
X_F^{\alpha}=\big[0,\ \min\{\tfrac45,\ 2\sqrt{3-2\alpha}-2\}\big],
\]

which shrinks as $\alpha$ grows, while $\lambda_\alpha=(1-\alpha)^{2}$ falls. Figure~\ref{fig:tradeoff} plots the two competing quantities over the admissible range $\alpha>1-1/\sqrt2$: raising $\alpha$ drives the contraction factor down, sharpening the rate, while the supply of admissible initial points shrinks. For the initial point $x_0=0.4$ we have $x_1=0.04$ and $x_0+x_1=0.44$, so $\alpha^{\ast}(x_0)=1-0.22=0.78$ by Proposition~\ref{prop:tradeoff}, giving $\lambda_{0.78}=0.0484$ and $M\lambda_{0.78}=0.0968$; the two constraints thus balance at an interior value, neither extreme of the admissible range being optimal. The ungraded choice, which must accept the weakest grade occurring along the orbit, corresponds to a threshold near $0.35$, where $\lambda_{0.35}=0.4225$ and $M\lambda_{0.35}=0.845$: admissible, but very much weaker.

\begin{table}[H]
\centering
\small
\caption{Observed gaps and envelopes along the orbit}
\label{tab:convergence}
\begin{tabular}{c c c c c c}
\hline
$n$ & $a_n$ & single-thr.\ $\alpha=0.78$ & single-thr.\ $\alpha=0.35$ & grade-adaptive \eqref{eq:adaptive} & grade $g_n$\\
\hline
$1$ & $1.296\cdot10^{-1}$ & $1.296\cdot10^{-1}$ & $1.296\cdot10^{-1}$ & $1.296\cdot10^{-1}$ & $0.7800$\\
$2$ & $1.568\cdot10^{-3}$ & $6.273\cdot10^{-3}$ & $5.476\cdot10^{-2}$ & $6.273\cdot10^{-3}$ & $0.9798$\\
$3$ & $1.600\cdot10^{-7}$ & $3.036\cdot10^{-4}$ & $2.313\cdot10^{-2}$ & $2.559\cdot10^{-6}$ & $0.9998$\\
$4$ & $1.600\cdot10^{-15}$ & $1.469\cdot10^{-5}$ & $9.774\cdot10^{-3}$ & $1.024\cdot10^{-13}$ & $1.0000$\\
$5$ & $1.600\cdot10^{-31}$ & $7.112\cdot10^{-7}$ & $4.130\cdot10^{-3}$ & $4.096\cdot10^{-29}$ & $1.0000$\\
\hline
\end{tabular}
\end{table}

Table~\ref{tab:convergence} lists the orbit generated from $x_0=0.4$ by the selection $x_{n+1}=x_n^{2}/4$, the successive gaps $a_n=d(x_{n-1},x_n)$, the single-threshold envelopes $\lambda_\alpha^{\,n-1}a_1$ of \eqref{eq:exactgap} at both thresholds, the grade-adaptive envelope $\big(\prod_{k<n}\lambda_{g_k}\big)a_1$ of \eqref{eq:adaptive}, and the grade $w(x_{n-1},x_n)$ realized along the orbit. The grades rise to $1$, so by Proposition~\ref{prop:adaptive} the adaptive envelope collapses far faster than either single-threshold envelope: at the fifth step it stands at $4.1\cdot10^{-29}$ against $7.1\cdot10^{-7}$ for the optimal single threshold, tracking the true gap $1.6\cdot10^{-31}$ to within a couple of orders of magnitude. Figure~\ref{fig:convergence} displays the same data logarithmically, together with the ascent of the grade along the orbit that drives the adaptive gain.

\end{example}

\section{Applications}\label{sec:applications}

The applications exercise the two novel ingredients separately: the set-valued contraction principle and the integral inclusion below put the asymmetric, relaxed-triangle metric to work with the graph switched off (complete graph), while the finite weighted network that follows puts the weighted graph to work, on a genuinely partial directed network where the grade-adaptive rate and the invariant $N(\alpha)$ are computed explicitly.

Taking the graph complete recovers and extends the classical statement.

\begin{proposition}\label{prop:complete}
Let $(X,d)$ be a generalized metric space satisfying \eqref{eq:M} with $(X,d_1)$ complete, and let $K\subseteq X$ be nonempty and closed in $(X,d_1)$. Let $F:K\to CB(K)$ satisfy
\[
H(Fx,Fy)\ \le\ \lambda\,d(x,y)\qquad(x,y\in K)
\]
for a constant $\lambda\in[0,1)$ with $M\lambda<1$. Then $F$ has a fixed point in $K$.
\end{proposition}

\begin{proof}
Let $G$ be the complete graph on $X$, loops included, graded $w\equiv1$, and put $\lambda_1:=\lambda$. Then $E_1=X\times X$, so $N(1)=1$ and $1\in A$; grade-compatibility is automatic because every pair is an edge of grade $1$; and property $(A_1)$ holds since $(x_n,x)\in E_1$ for every $n$ and every limit. Theorem~\ref{thm:main} applies with $\alpha=1$, started at any $x_0\in K$ with a selected $x_1\in F(x_0)$.
\end{proof}

For $M=1$ and $d$ a metric, Proposition~\ref{prop:complete} is the classical set-valued contraction principle \cite{Nadler1969} for the one-sided excess; for $d$ a symmetric $b$-metric with constant $M=s$ it is the corresponding $b$-metric statement \cite{Czerwik2022}, the requirement $M\lambda<1$ being the familiar $s\lambda<1$; and for asymmetric $d$ it is new. The three cases are thus unified.

Let $X=C[0,1]$ and, for a fixed $p\ge1$, set
\[
d(x,y):=\|x-y\|_\infty^{\,p},\qquad\|u\|_\infty:=\sup_{t\in[0,1]}|u(t)| .
\]
Since $r\mapsto r^{p}$ is a homeomorphism of $[0,\infty)$, the topology of $d$ is that of uniform convergence and $d_1=\|\cdot\|_\infty$ induces it, so $(X,d_1)$ is complete; convexity of $r\mapsto r^{p}$ gives $|a+b|^{p}\le2^{p-1}(|a|^{p}+|b|^{p})$, so \eqref{eq:M} holds with $M=2^{\,p-1}\ge1$. Thus $(C[0,1],d)$ is a generalized metric space, an ordinary metric space when $p=1$ and a genuine $b$-metric space when $p>1$. Note that $K:=C[0,1]$ is unbounded in $d$, which is admissible by Remark~\ref{rem:unbounded}.

Let $\Gamma:[0,1]\times\R\to\R$ have nonempty compact values and consider the integral inclusion
\begin{equation}\label{eq:inclusion}
x(t)\ \in\ \int_0^{t}\Gamma\big(s,x(s)\big)\,ds,\qquad t\in[0,1],
\end{equation}
a solution being $x\in C[0,1]$ for which some integrable selection $g(s)\in\Gamma(s,x(s))$ satisfies $x(t)=\int_0^{t}g(s)\,ds$. Assume:
\begin{enumerate}[label=\textnormal{(\roman*)}]
\item $s\mapsto\Gamma(s,u)$ is measurable for each $u$, and $\sup\{|z|:z\in\Gamma(s,u)\}\le c(1+|u|)$ for a constant $c>0$;
\item $H_\R\big(\Gamma(s,u),\Gamma(s,v)\big)\le\kappa|u-v|$ for all $s\in[0,1]$ and $u,v\in\R$, where $H_\R$ is the Hausdorff distance on compact subsets of $\R$;
\item $2^{\,p-1}\kappa^{\,p}<1$, that is, $\kappa<2^{(1-p)/p}$.
\end{enumerate}

\begin{theorem}\label{thm:inclusion}
Under \textnormal{(i)--(iii)} the integral inclusion \eqref{eq:inclusion} admits a solution in $C[0,1]$.
\end{theorem}

\begin{proof}
For $x\in C[0,1]$ let $S(x)$ be the set of measurable selections $g$ with $g(s)\in\Gamma(s,x(s))$ almost everywhere; $S(x)$ is nonempty by the Kuratowski--Ryll-Nardzewski selection theorem, and by (i) every such $g$ obeys $|g(s)|\le c(1+\|x\|_\infty)$ and is therefore integrable. Define
\[
F(x):=\Big\{y\in C[0,1]:\ y(t)=\int_0^{t}g(s)\,ds\ \text{for some}\ g\in S(x)\Big\}.
\]
Each $y\in F(x)$ satisfies $|y(t)|\le c(1+\|x\|_\infty)$ and $|y(t)-y(t')|\le c(1+\|x\|_\infty)|t-t'|$, so $F(x)$ is bounded and equicontinuous, hence relatively compact by the Arzel\`a--Ascoli theorem; closedness follows from compactness of the values of $\Gamma$. Thus $F(x)$ is a nonempty compact, in particular closed and bounded, subset of $C[0,1]$, so $F:K\to CB(K)$.

Let $x,x'\in C[0,1]$ and $y\in F(x)$ with selection $g\in S(x)$. By (ii), for almost every $s$ the set $\{z\in\Gamma(s,x'(s)):|g(s)-z|\le\kappa|x(s)-x'(s)|\}$ is nonempty, and admits a measurable selection $g'\in S(x')$ with $|g(s)-g'(s)|\le\kappa|x(s)-x'(s)|\le\kappa\|x-x'\|_\infty$ almost everywhere. Putting $y'(t):=\int_0^{t}g'(s)\,ds\in F(x')$,
\[
\|y-y'\|_\infty=\sup_{t}\Big|\int_0^{t}\big(g(s)-g'(s)\big)ds\Big|\le\int_0^{1}|g(s)-g'(s)|\,ds\le\kappa\|x-x'\|_\infty .
\]
Taking the supremum over $y\in F(x)$ gives $H_\infty(Fx,Fx')\le\kappa\|x-x'\|_\infty$ for the excess in the sup-metric, and since $r\mapsto r^{p}$ is increasing,
\[
H(Fx,Fx')=H_\infty(Fx,Fx')^{\,p}\le\kappa^{\,p}\|x-x'\|_\infty^{\,p}=\kappa^{\,p}\,d(x,x').
\]
So $F$ is a set-valued contraction on $(C[0,1],d)$ with constant $\lambda=\kappa^{\,p}$, and (iii) states exactly $M\lambda=2^{\,p-1}\kappa^{\,p}<1$. Propoition ~\ref{prop:complete} yields $x^\ast\in F(x^\ast)$, which is precisely a solution of \eqref{eq:inclusion}.
 
\end{proof}

The role of the constant $M$ is visible in (iii): for $p>1$ the underlying space is a genuine $b$-metric space with $M=2^{\,p-1}>1$, and the admissible Lipschitz constant $\kappa$ is reduced accordingly relative to the classical case $p=1$, $M=1$, where the requirement is simply $\kappa<1$.

We now put the weighted graph to genuine use on a finite directed network that is not complete, computing the invariant $N(\alpha)$, the admissible set $A$, and the grade-adaptive rate.
\begin{example}Let $X=\R$ with $d(x,y)=|x-y|$, so $M=1$, and let
\[
K=\{0,1,2,3,6\}.
\]
Define the branching set-valued map
\[
F(6)=\{2,3\},\quad F(3)=\{1\},\quad F(2)=\{0\},\quad F(1)=\{0\},\quad F(0)=\{0\},
\]
and let $E(G)$ consist of the loops together with the directed edges $(6,3),(6,2),(3,1),(2,0),(1,0)$ and their reversals, graded
\[
w(6,3)=w(6,2)=0.3,\qquad w(3,1)=w(2,0)=0.6,\qquad w(1,0)=1,
\]
loops graded $1$. The network and its operative level cut are shown in Figure~\ref{fig:digraph}. Take the family $\lambda_{0.3}=0.8$, $\lambda_{0.6}=0.5$, $\lambda_1=0$ (non-increasing, and $M\lambda_{0.3}=0.8<1$). Table~\ref{tab:digraph} verifies the contractive estimate on every non-loop edge; the largest excess ratio is $0.75$, on the weakest edge $(6,2)$, and each ratio is dominated by the factor $\lambda$ carried by that edge's grade.

\begin{figure}[H]
\centering
\caption{}
\label{fig:digraph}
\begin{tikzpicture}[
  every node/.style={circle,draw,inner sep=1.4pt,minimum size=7mm,font=\small},
  edg/.style={-{Stealth[length=2mm]},thick},
  gone/.style={-{Stealth[length=2mm]},draw=black!22,thick},
  lab/.style={draw=none,circle=false,inner sep=1pt,font=\scriptsize,fill=white}
]
\begin{scope}
\node (s6) at (0,0) {$6$};
\node (s3) at (2,0.9) {$3$};
\node (s2) at (2,-0.9) {$2$};
\node (s1) at (4,0.9) {$1$};
\node (s0) at (6,0) {$0$};
\draw[edg] (s6)--(s3) node[lab,midway,above left]{$0.3$};
\draw[edg] (s6)--(s2) node[lab,midway,below left]{$0.3$};
\draw[edg] (s3)--(s1) node[lab,midway,above]{$0.6$};
\draw[edg] (s1)--(s0) node[lab,midway,above right]{$1$};
\draw[edg] (s2)--(s0) node[lab,midway,below right]{$0.6$};
\node[draw=none,font=\small] at (3,-1.9) {(a) $(G,w)$};
\end{scope}
\begin{scope}[xshift=8.4cm]
\node (t6) at (0,0) {$6$};
\node (t3) at (1.7,0.9) {$3$};
\node (t2) at (1.7,-0.9) {$2$};
\node (t1) at (3.4,0.9) {$1$};
\node (t0) at (5.1,0) {$0$};
\draw[gone] (t6)--(t3);
\draw[gone] (t6)--(t2);
\draw[edg] (t3)--(t1);
\draw[edg] (t1)--(t0);
\draw[edg] (t2)--(t0);
\node[draw=none,font=\small] at (2.55,-1.9) {(b) $G_\alpha$, $\alpha=0.6$:\ $6$ severed};
\end{scope}
\end{tikzpicture}
\end{figure}

\begin{table}[H]
\centering
\caption{Verification of the contractive estimate on the network}
\label{tab:digraph}
\begin{tabular}{c c c c c c}
\hline
edge $(x,y)$ & $F(x),\,F(y)$ & $H(Fx,Fy)$ & $d(x,y)$ & $w(x,y)$ & bound $\lambda_{w(x,y)}\,d(x,y)$\\
\hline
$(6,3)$ & $\{2,3\},\,\{1\}$ & $2$ & $3$ & $0.3$ & $0.8\cdot3=2.4$\\
$(6,2)$ & $\{2,3\},\,\{0\}$ & $3$ & $4$ & $0.3$ & $0.8\cdot4=3.2$\\
$(3,1)$ & $\{1\},\,\{0\}$ & $1$ & $2$ & $0.6$ & $0.5\cdot2=1.0$\\
$(2,0)$ & $\{0\},\,\{0\}$ & $0$ & $2$ & $0.6$ & $0.5\cdot2=1.0$\\
$(1,0)$ & $\{0\},\,\{0\}$ & $0$ & $1$ & $1$ & $0\cdot1=0$\\
\hline
\end{tabular}
\end{table}

The level cuts are governed by the two non-trivial grades. For $\alpha\le0.3$ all five edges survive and $\widetilde G_\alpha$ is connected with diameter $N(\alpha)=2$; for $0.3<\alpha\le0.6$ the two grade-$0.3$ edges out of the vertex $6$ are deleted, leaving $6$ with no non-loop edge, so $\widetilde G_\alpha$ is disconnected and $N(\alpha)=\infty$; the same holds for $\alpha>0.6$. Hence
\[
A=(0,\,0.3],\qquad N(\alpha)=2\ \text{on}\ A .
\]
So the only globally admissible threshold is $\alpha=0.3$, at which the plain rate \eqref{eq:exactgap} guarantees no better than $\lambda_{0.3}^{\,n}=0.8^{\,n}$ -- a slow geometric decay that never certifies termination. The grade-adaptive rate is another matter. Starting from $x_0=6$ and selecting $x_1=3$, the orbit
\[
6\ \longrightarrow\ 3\ \longrightarrow\ 1\ \longrightarrow\ 0
\]
has edge grades $g_1=w(6,3)=0.3$, $g_2=w(3,1)=0.6$, $g_3=w(1,0)=1$, ascending as Proposition~\ref{prop:adaptive} requires, with gaps $a_1=3,a_2=2,a_3=1$. The adaptive bound \eqref{eq:adaptive} gives
\[
a_2\le\lambda_{0.3}a_1=2.4,\qquad a_3\le\lambda_{0.3}\lambda_{0.6}a_1=1.2,\qquad a_4\le\lambda_{0.3}\lambda_{0.6}\lambda_1 a_1=0,
\]
correctly certifying that the fixed point $0$ is reached at the third step, whereas the plain bound only offers $a_4\le0.8^{3}\cdot3=1.536$. The other selection $x_1=2$ produces the shorter orbit $6\to2\to0$, again terminating, with grades $0.3,0.6$; the two branches illustrate the non-uniqueness of the orbit intrinsic to the set-valued setting. This is the sense in which, on a network where the weakest link forces a poor uniform threshold, the weights nonetheless deliver a sharp, indeed exact, guarantee once the iteration climbs into the strongly graded interior.

\end{example}

\section{Conclusion}\label{sec:conclusion}

We obtained the set-valued contraction principle into generalized metric spaces carrying a weighted graph. The distance may be asymmetric and satisfy only a relaxed triangle inequality controlled by a constant $M$, while the contraction is imposed solely along graph edges and may depend on the grade each edge carries. Theorem~\ref{thm:main} establishes the existence of a fixed point together with an explicit geometric rate governed by the product $M\lambda_\alpha<1$; the boundedness of the underlying set is not required, which allows the theory to be applied on unbounded function spaces. Proposition~\ref{prop:adaptive} shows that the weighting is not merely bookkeeping but sharpens the rate: because the grades encountered along the orbit ascend, the contraction factors actually applied decrease, and the resulting estimate can outrun any single-threshold bound and certify exact termination. Proposition~\ref{prop:tradeoff} shows that the entry threshold itself is subject to a genuine trade-off -- a higher threshold sharpens the guarantee but depletes the supply of initial points -- whose optimum is in general interior, as Example~\ref{ex:interval} and Figure~\ref{fig:tradeoff} demonstrate quantitatively. Proposition~\ref{prop:nonunique} shows that there is no uniqueness. Corollary~\ref{cor:finite} shows that on finite networks the approximation terminates after finitely many steps, and the applications recover the classical and $b$-metric statements, solve a class of integral inclusions, and, on a partial directed network, exhibit the weighted machinery and the invariant $N(\alpha)$ in explicit computation. In future we plan to further explore; First a chain-based estimate valid uniformly over all pairs of initial points. Second to admit contraction data varying continuously with the edge rather than through finitely many grade levels, and to characterize the optimal threshold by a first-order condition rather than by comparison of finitely many branches.

\section*{Declarations}

\noindent\textbf{Conflict of interest.} The authors declare that they have no conflict of interest.

\noindent\textbf{Data availability.} No datasets were generated or analysed during the current study. All numerical values reported in Tables~\ref{tab:convergence} \& \ref{tab:digraph} and Figures~\ref{fig:convergence} \& \ref{fig:tradeoff}  are computed directly from the closed-form data of the corresponding examples.

\noindent\textbf{Funding.} The authors received no specific funding for this work.

\noindent\textbf{Author contributions.} Both authors contributed equally to the conception, the proofs and the writing, and both approved the final manuscript.

\end{document}